\documentclass[11pt,twoside]{article}
\usepackage{amsmath, amsthm, amscd, amsfonts, amssymb, graphicx, color}
\usepackage[bookmarksnumbered, colorlinks]{hyperref} \usepackage{float}
\usepackage{lipsum}
\usepackage{afterpage}
\usepackage[labelfont=bf]{caption}
\usepackage[nottoc,notlof,notlot]{tocbibind} 

\usepackage[noadjust]{cite}
\usepackage{lipsum}
\usepackage{fancyhdr}
\newtheorem{theorem}{Theorem}[section]

\newtheorem{proposition}[theorem]{Proposition}
\newtheorem{corollary}[theorem]{Corollary}
\theoremstyle{definition}
\newtheorem{definition}[theorem]{Definition}

\newtheorem{remark}[theorem]{Remark}
\renewenvironment{proof}{{\bfseries \noindent Proof.}}{~~~~$\square$}
\makeatletter
\def\th@newremark{\th@remark\thm@headfont{\bfseries}}
\makeatletter

\begin{document}
\thispagestyle{plain}
‎
\begin{center}

{\Large \bf 
Spherical Indicatrices of a Bertrand Curve in Three Dimensional Lie Groups\\}

\let\thefootnote\relax\footnote

{\bf Ali \c{C}akmak}\vspace*{-2mm}\\
\vspace{2mm} {\small  Bitlis Eren University} \vspace{2mm}

\end{center}

\vspace{4mm}

{\footnotesize
\begin{quotation}
{\noindent \bf Abstract.} In this paper, new representations of a Bertrand curve pair in three
dimensional Lie groups with bi-invariant metric are given. Besides, the
spherical indicatrices of a Bertrand curve pair are obtain and the relations between the
spherical indicatrices and new representations of Bertrand curve pair are shown. 
\end{quotation}
\begin{quotation}

\end{quotation}}

\section{Introduction}
\label{intro} 
The curve theory has an important place in differential geometry. Many
mathematicians made a significant contribution to this subject from past to
present. Among these contributions, some of the relationships among the curve
pairs are particularly interesting. The curve pairs for which there exits some
relationships between their Frenet vectors or curvatures are examples to
special space curves. In particular, Involute-evolute curves, Bertrand curves,
Mannheim curves are well-known examples of curve pairs and have been studied
by many mathematicians \cite{1}, \cite{6,10}.

It is know that the curve $\alpha$, which is given by a parametrized regular
curve, is called a Bertrand curve if there exists another curve $\beta$ such
that the principal normal lines to $\alpha$ and $\beta$ are equal at all
points. The curve $\beta$ is called a Bertrand mate of $\alpha$ \cite{14}. These curves
first investigated by J. Bertrand. He proved that the curve $\alpha$ for which
there exists a linear relation between curvature and torsion: $a\kappa
+b\tau=c,$ $(a,b,c$ are non-zero constants$)$ admits a Bertrand mate.

The degenerate semi-Riemannian geometry of a Lie group are examined in \cite{2}. The general helices in three dimensional Lie groups with a
bi-invariant metric are introduced and a generalization of Lancret's theorem
is obtained in \cite{4}. The some special curves in three dimensional Lie
groups using harmonic curvature function are recharacterized and Bertrand
curves are investigated via harmonic curvature function in three dimensional
Lie groups in \cite{11}. The some relations between slant helices and their
involutes, spherical images in three dimensional Lie groups are given in
\cite{12}. The some features of the spherical indicatrices of a Bertrand curve
and its mate curve are presented in Euclidean 3-space in \cite{13}.

In this paper, we have done a study on spherical indicatrices of Bertrand
curve pair in three dimensional Lie groups with a bi-invariant metric. As far as
we know, spherical indicatrices of a Bertrand curve and its mate curve have
not been shown in three dimensional Lie groups. Hence, this study is intended
to fill this gap.

\section{Preliminaries}
\label{sec:2}
Suppose that $G$ is a Lie group with a bi-invariant metric such that
$\left\langle ,\right\rangle $ is a bi-invariant metric on $G.$ If the Lie
algebra of $G$ is given by $g,$ the Lie algebra $g$ is isomorphic to $T_{e}G,$
where $e$ is neutral element of $G$. Since $\left\langle ,\right\rangle $ is a
bi-invariant metric on $G,$ we get
\begin{equation}
\left\langle X,[Y,Z]\right\rangle =\left\langle [X,Y],Z\right\rangle ,
\label{1}%
\end{equation}
and
\begin{equation}
D_{x}Y=\frac{1}{2}[X,Y], \label{2}%
\end{equation}
where $X,Y,Z\in g$ and $D$ is the Levi-Civita connection of $G.$

Let us assume that $\alpha:I\subset{\mathbb{R}}\rightarrow G$ is an
arc-lenghted curve and $\{X_{1},X_{2},...,X_{n}\}$ is an orthonormal basis of
$g$. Here, we consider that any two vector fields $W$ and $Z$ along the curve
$\alpha$ as $W=\sum_{i=1}^{n}w_{i}X_{i}$ and $Z=\sum_{i=1}^{n}z_{i}X_{i}$ such
that $w_{i}:I\rightarrow{\mathbb{R}}$ and $z_{i}:I\rightarrow{\mathbb{R}}$ are
smooth functions. It is well-known that Lie bracket of $W$ and $Z$ can be
written as
\[
\lbrack W,Z]=\sum_{i=1}^{n}w_{i}z_{i}[X_{i},X_{j}],
\]
and $D_{\alpha^{\prime}}W$ is obtained as%
\begin{equation}
D_{\alpha^{\prime}}W=\dot{W}+\frac{1}{2}[T,W], \label{3}%
\end{equation}
where $T=\alpha^{\prime}$, $\dot{W}=\sum_{i=1}^{n}\dot{w}_{i}X_{i}$ and
$D_{\alpha^{\prime}}W$ is the covariant derivative of $W$ along the curve
$\alpha.$ In this case, $\dot{W}=0$ under the condition that $W$ is the
left-invariant vector field to the curve $\alpha$ \cite{3}.

Suppose that $G$ is a Lie group and $(T,N,B,\kappa,\tau)$ is the Frenet
apparatus of the curve $\alpha$. Hence, the Frenet formulas of the curve
$\alpha$ can be written as
\[
D_{T}T=\kappa N,\,\,\,D_{T}N=-\kappa T+\tau B,\,\,\,D_{T}B=-\tau N,
\]
where $\kappa=\left\Vert \dot{T}\right\Vert $ \cite{3}.

\begin{proposition}
	\textbf{\cite{4}\textit{ }}\textit{ Suppose that the curve $\alpha(s)$ is a
		curve\ in Lie group }$\mathit{G}$\textit{ such that the parameter }%
	$\mathit{s}$\textit{ is the arc length parameter of $\alpha(s)$ and\ the
		Frenet apparatus of $\alpha(s)$ are$\ (T,N,B,\kappa,\tau)$. Then the following
		equalities hold}
	\begin{equation}
	\left\{
	\begin{array}
	[c]{l}%
	{[T,N]=\left\langle [T,N],\,\,B\right\rangle B=2\tau_{G}B,}\\
	{[T,\,\,B]=\left\langle [T,B],\,\,N\right\rangle N=-2\tau_{G}N.}%
	\end{array}
	\right.  \label{4}%
	\end{equation}
	
\end{proposition}

Suppose that the curve $\alpha(s)$ is a curve in three dimensional Lie group
$G$ such that the parameter $s$ is the arc length parameter of $\alpha(s)$.
Then from the Eq.(3) and Proposition 1, Frenet formulas are found as follows:%
\[
\left(
\begin{array}
[c]{c}%
{\frac{dT}{ds}}\\
{\frac{dN}{ds}}\\
{\frac{dB}{ds}}%
\end{array}
\right)  =\left(
\begin{array}
[c]{ccc}%
{0} & {\kappa} & {0}\\
{-\kappa} & {0} & {\tau-\tau_{G}}\\
{0} & {-(\tau-\tau_{G})} & {0}%
\end{array}
\right)  \left(
\begin{array}
[c]{c}%
{T}\\
{N}\\
{B}%
\end{array}
\right)  ,
\]
where $\left\{  T,N,B\right\}  $ is the Frenet frame, $\tau_{G}=\frac{1}%
{2}\left\langle [T,N],B\right\rangle $ and $\kappa,\tau$ are curvature and
torsion of $\alpha$ in $G$ , respectively \cite{12}.

\begin{definition}
	\textbf{\cite{4} }Suppose that the curve $\alpha(s)$ is a curve in three
	dimensional Lie group $G$ such that the parameter $s$ is the arc length
	parameter of $\alpha(s)$ and the Frenet apparatus of $\alpha(s)$ are
	$(T,N,B,\kappa,\tau)$. Then the harmonic curvature function of the curve
	$\alpha$ can be given by
	\begin{equation}
	H=\frac{\tau-\tau_{G}}{\kappa}. \label{5}%
	\end{equation}
	
\end{definition}

\begin{theorem}
	\textbf{\textit{ }\cite{4}}\textit{ Suppose that the curve }$\alpha
	(s)$\textit{ is a curve\ in Lie group }$\mathit{G}$\textit{ such that the
		parameter }$\mathit{s}$ is the arc length parameter of $\alpha(s)$ and\ the
	Frenet apparatus of $\alpha(s)$ are$\ (T,N,B,\kappa,\tau)$.\textit{The curve
	}$\alpha$\textit{ is a general helix if and only if }$\tau=c\kappa+\tau_{G}%
	$\textit{ , where }$c\in{\mathbb{R}}$\textit{. }
\end{theorem}

\noindent

Definition 2 and Theorem 3 immediately give the following corollary:

\begin{corollary}
	\noindent\textit{Suppose that the curve }$\alpha$\textit{ is a curve\ in
	}$\mathit{G.}$ \textit{Being a general helix in $G$ of }$\alpha$ \textit{is a
	necessary and sufficient condition of being $H=\mathrm{constant}$.}
\end{corollary}

\begin{theorem}
	\textbf{\textit{ }\cite{12}}\textit{ Suppose that $\alpha$ is an arc length
		parametrized curve with the Frenet apparatus $(T,N,B,\kappa,\tau)$ in $G$ and
		$H$ is the harmonic curvature function of the curve $\alpha.$Then $\alpha$ is
		a slant helix if and only if the function }
	\begin{equation}
	\sigma=\frac{\kappa(1+H^{2})^{3/2}}{H^{\prime}}, \label{6}%
	\end{equation}
	\textit{is a constant.}
\end{theorem}

\begin{theorem}
	\cite{11}\textit{ Suppose that the curve $\alpha(s)$ is a Bertrand curve in
		$G$ such that the parameter }$s$ is \textit{the arc length parameter of
		$\alpha(s)$ and the Frenet apparatus of $\alpha(s)$ are }$(T,N,B,\kappa,\tau
	)$.\textit{ For all $s\in I$, $\alpha$ satisfy the equation $\lambda
		\kappa(s)+\mu\kappa(s)H(s)=1$, where $\lambda,\mu$ are constants and $H$ is
		harmonic curvature function of $\alpha(s)$. If the curve $\beta(s)$ is given
		by $\beta(s)=\alpha(s)+\lambda N(s)$ for all $s\in I$, then $(\alpha,\beta)$
		is the Bertrand curve pair.}
\end{theorem}

\begin{theorem}
	\textbf{\cite{11} }Suppose that \textit{$(\mathit{\alpha,\beta)}$ is a
		Bertrand curve pair in $G.$ Then, being slant helix of $\alpha$ is a necessary
		and sufficient condition of being a slant helix of }$\beta$.
\end{theorem}

\begin{remark}
	\textbf{ \cite{5} }Suppose that $G$ is a Lie group with a bi-invariant metric
	such that $\left\langle ,\right\rangle $ is a bi-invariant metric on $G.$
	Hence, the following items can be written in various Lie groups:
	
	\textit{i)} If $G$ is abelian group, $\tau_{G}=0$.
	
	\textit{ii)} If $G$ is $SO^{3},$ $\tau_{G}=\frac{1}{2}$.
	
	\textit{iii)} If $G$ is $SU^{3},$ $\tau_{G}=1$.
\end{remark}

\section{Spherical Indicatrices of a Bertrand Curve in Three Dimensional Lie Groups}
In this chapter, the spherical indicatrices of a Bertrand curve with respect
to its partner curve in $G$ with a bi-invariant metric are presented and given
some significant results by using the features of the curves.

\begin{theorem}
	Suppose that $(\alpha,\tilde{\alpha})$ is a Bertrand curve pair with
	arc-length parameter $s$ and $s^{\ast}$, respectively and the Frenet
	invariants of $\alpha$ and $\tilde{\alpha}$ are denoted by $\{T,N,B,\kappa
	,\tau-\tau_{G}\}$ and $\{\tilde{T},\tilde{N},\tilde{B},\tilde{\kappa}%
	,\tilde{\tau}-\tilde{\tau}_{G}\}$, respectively. Then the relationship between
	the Frenet invariants of $\alpha$ and $\tilde{\alpha}$ is given as follows:%
	\[
	T=\frac{-1}{\sqrt{1+\rho^{2}}}\{\tilde{T}-\rho\tilde{B}\},\text{\ }%
	N=\epsilon\tilde{N},\text{\ }B=\frac{-\epsilon}{\sqrt{1+\rho^{2}}}\{\rho
	\tilde{T}+\tilde{B}\}
	\]
	and%
	\[
	\kappa=\frac{-\epsilon\tilde{\kappa}(1+\tilde{H}\rho)(\rho-\tilde{H})}%
	{\tilde{H}(1+\rho^{2})},\text{ }\tau-\tau_{G}=\frac{\tilde{\kappa}(\rho
		-\tilde{H})^{2}}{\tilde{H}(1+\rho^{2})},
	\]
	where $\tilde{H}$ is a harmonic curvature function of the curve $\tilde
	{\alpha}$ and $\rho=\frac{(\tilde{\tau}-\tilde{\tau}_{G})^{^{\prime}}}%
	{\tilde{\kappa}^{^{\prime}}}$.
\end{theorem}

\begin{proof}
	Since $\alpha$ and $\tilde{\alpha}$ are the Bertrand curves, then%
	\begin{equation}
	\alpha(s)=\tilde{\alpha}(s^{\ast})-\lambda\epsilon\tilde{N} \label{7}%
	\end{equation}
	Differentiating the Eq. (7) according to s and using the Eq. (3), we have%
	\begin{equation}
	T\frac{ds}{ds^{\ast}}=\frac{d\tilde{\alpha}(s^{\ast})}{ds}-\lambda
	\epsilon(D_{\tilde{T}}\tilde{N}-\frac{1}{2}\left[  \tilde{T},\tilde{N}\right]
	) \label{8}%
	\end{equation}
	By using the Frenet formulas, we get%
	\begin{equation}
	T\frac{ds}{ds^{\ast}}=\tilde{T}(1+\lambda\epsilon\tilde{\kappa})-\lambda
	\epsilon(\tilde{\tau}-\tilde{\tau}_{G})\tilde{B} \label{9}%
	\end{equation}
	which gives us%
	\begin{equation}
	\frac{ds^{\ast}}{ds}=\frac{1}{\sqrt{(1+\lambda\epsilon\tilde{\kappa}%
			)^{2}+\lambda^{2}(\tilde{\tau}-\tilde{\tau}_{G})^{2}}}\text{.} \label{10}%
	\end{equation}
	From the Eq. (9) and (10), we have%
	\begin{equation}
	T=\frac{1}{\sqrt{(1+\lambda\epsilon\tilde{\kappa})^{2}+\lambda^{2}(\tilde
			{\tau}-\tilde{\tau}_{G})^{2}}}\tilde{T}(1+\lambda\epsilon\tilde{\kappa
	})-\lambda\epsilon(\tilde{\tau}-\tilde{\tau}_{G})\tilde{B} \label{11}%
	\end{equation}
	Since $B=T\times N$ , from (11) it is obtained that%
	\begin{equation}
	B=\frac{1}{\sqrt{(1+\lambda\epsilon\tilde{\kappa})^{2}+\lambda^{2}(\tilde
			{\tau}-\tilde{\tau}_{G})^{2}}}\tilde{T}(\lambda\epsilon(\tilde{\tau}%
	-\tilde{\tau}_{G}))+(1+\lambda\epsilon\tilde{\kappa})\tilde{B} \label{12}
	\end{equation}
	By taking the derivative of Eq.(11) according to $s$ and considering Frenet
	formulas, we have
	\begin{equation}
	\begin{aligned}
	T^{^{\prime}}\frac{ds}{ds^{\ast}}  &  =\left[  \left(  \lambda\epsilon
	\tilde{\kappa}^{^{\prime}}\right)  ((1+\lambda\epsilon\tilde{\kappa}%
	)^{2}+\lambda^{2}(\tilde{\tau}-\tilde{\tau}_{G})^{2})\right] \nonumber\\
	&  -\left[  (1+\lambda\epsilon\tilde{\kappa})\left[  (1+\lambda\epsilon
	\tilde{\kappa})\lambda\epsilon\tilde{\kappa}^{^{\prime}}+\lambda^{2}%
	(\tilde{\tau}-\tilde{\tau}_{G})^{^{\prime}}(\tilde{\tau}-\tilde{\tau}%
	_{G})\right]  \right]  \tilde{T}+\\
	&  \left[  \left(  \tilde{\kappa}(1+\lambda\epsilon\tilde{\kappa}%
	)+\lambda\epsilon(\tilde{\tau}-\tilde{\tau}_{G})^{2}\right)  ((1+\lambda
	\epsilon\tilde{\kappa})^{2}+\lambda^{2}(\tilde{\tau}-\tilde{\tau}_{G}%
	)^{2})\right]  \tilde{N}+\\
	&  \left[
	\begin{array}
	[c]{c}%
	\left(  -\lambda\epsilon(\tilde{\tau}-\tilde{\tau}_{G})^{^{\prime}}\right)
	((1+\lambda\epsilon\tilde{\kappa})^{2}+\lambda^{2}(\tilde{\tau}-\tilde{\tau
	}_{G})^{2})+\\
	\left[  (1+\lambda\epsilon\tilde{\kappa})\left[  (\lambda\epsilon(\tilde{\tau
	}-\tilde{\tau}_{G}))\lambda\epsilon\tilde{\kappa}^{^{\prime}}+\lambda
	^{2}(\tilde{\tau}-\tilde{\tau}_{G})^{^{\prime}}(\tilde{\tau}-\tilde{\tau}%
	_{G})\right]  \right]
	\end{array}
	\right]  \tilde{B}
	\end{aligned}
	\end{equation}
	we know that $N$ and $\tilde{N}$ are linearly dependent. Then from the last equation we have%
	\begin{equation}\begin{aligned}
		&\left[  (\lambda\epsilon\tilde{\kappa}^{^{\prime}})((1+\lambda\epsilon
	\tilde{\kappa})^{2}+\lambda^{2}(\tilde{\tau}-\tilde{\tau}_{G})^{2})\right] \\
	&-\left[  (1+\lambda\epsilon\tilde{\kappa})\left[  (1+\lambda\epsilon
	\tilde{\kappa})\lambda\epsilon\tilde{\kappa}^{^{\prime}}+\lambda^{2}%
	(\tilde{\tau}-\tilde{\tau}_{G})^{^{\prime}}(\tilde{\tau}-\tilde{\tau}%
	_{G})\right]  \right]  =0\text{.}
	\end{aligned}
	\end{equation}
	From the last equation, we get%
	\begin{equation}
	\lambda=\frac{-\epsilon\rho}{\tilde{\kappa}(\rho-\tilde{H})} \label{17}%
	\end{equation}
	where $\rho=\frac{(\tilde{\tau}-\tilde{\tau}_{G})^{^{\prime}}}{\tilde{\kappa
		}^{^{\prime}}}$ and $\tilde{H}=\frac{\tilde{\tau}-\tilde{\tau}_{G}}%
	{\tilde{\kappa}}$.
	
	If substituting the Eq. (14) in the Eq. (11) and (12), we have
	\begin{align}
	T  &  =\frac{-1}{\sqrt{1+\rho^{2}}}\{\tilde{T}-\rho\tilde{B}\}\text{,
	}\label{18}\\
	\text{\ }B  &  =\frac{-\epsilon}{\sqrt{1+\rho^{2}}}\{\rho\tilde{T}+\tilde{B}\}
	\label{19}%
	\end{align}
	By using Frenet formulas and with the help of the Proposition 2.1, we get
	$T^{^{\prime}}=\kappa N$. \ By considering the Eq. (15) we obtain%
	\begin{equation}
	\kappa=\frac{-\epsilon\tilde{\kappa}(1+\tilde{H}\rho)(\rho-\tilde{H})}%
	{\tilde{H}(1+\rho^{2})}\text{.} \label{20}%
	\end{equation}
	Since $\tau^{^{\prime}}=\left\langle B^{^{\prime}},N\right\rangle $, we get
	$\tau-\tau_{G}=\frac{\tilde{\kappa}(\rho-\tilde{H})^{2}}{\tilde{H}(1+\rho
		^{2})}$. Hence, our theorem is proved.
\end{proof}

Thus, the geodesic curvature of the principal image of the principal normal
indicatrix of $\alpha$ is given by%
\begin{equation}
\Gamma=\frac{-\tilde{\kappa}(\rho-\tilde{H})}{\tilde{\kappa}^{2}(1+\tilde
	{H}^{2})^{\frac{3}{2}}}\text{.} \label{21}%
\end{equation}

\begin{corollary}
	$\alpha$ and $\tilde{\alpha}$ be a Bertrand curve mate with arc-length
	parameter $s$ and $s^{\ast}$, respectively. Then the relationship between arc
	length parameters $s$ and $s^{\ast}$ is given by
\end{corollary}

\begin{equation}
s=\int\frac{\tilde{H}\sqrt{1+\rho^{2}}}{\rho-\tilde{H}}ds^{\ast} \label{22}%
\end{equation}
\subsection{Tangent indicatrix $\alpha_{t}$ $=T$ of the Bertrand curve in $G$}
\begin{definition}
	Suppose that the curve $\alpha$ is a regular curve in $G$. The curve
	$\alpha_{t}$ is called the tangent indicatrix of the curve $\alpha$ and
	$\alpha_{t}:I\subset S^{2}\subset g$ is defined by%
	\[
	\alpha_{t}(s_{t})=T(s)\text{.}%
	\]
	Then the tangent indicatrix of a Bertrand curve in three dimensional Lie group
	can be given by%
	\begin{equation}
	\alpha_{t}=\frac{-1}{\sqrt{1+\rho^{2}}}\{\tilde{T}-\rho\tilde{B}\}\text{.}
	\label{23}%
	\end{equation}
	
\end{definition}

Thus, we can give the following characterizations in the view of above equation.
\begin{theorem}
	Suppose that the curve $\alpha$ is a regular curve in $G$ and the Frenet
	apparatus of the tangent indicatrix $\alpha_{t}$ $=T$ of the Bertrand curve
	are denoted by $\{T_{t},N_{t},B_{t},\kappa_{t},(\tau-\tau_{G})_{t}\}$. Then we
	have%
	\begin{equation}
	T_{t}=-\tilde{N}\text{, }N_{t}=\frac{1}{\sqrt{1+\tilde{H}^{2}}}\{\tilde
	{T}-\tilde{H}\tilde{B}\}\text{, }B_{t}=\frac{1}{\sqrt{1+\tilde{H}^{2}}%
	}\{\tilde{H}\tilde{T}+\tilde{B}\}\text{\ } \label{24}%
	\end{equation}
	and%
	\begin{equation}
	s_{t}=-%
	{\displaystyle\int}
	\frac{\tilde{\kappa}(\rho-\tilde{H})^{2}}{\tilde{H}(1+\rho)^{2}}ds\text{,
	}\kappa_{t}=\frac{\sqrt{1+\rho^{2}}\sqrt{1+\tilde{H}^{2}}}{\tilde{H}-\rho
}\text{, }(\tau-\tau_{G})_{t}=\frac{-\tilde{\kappa}^{^{\prime}}\sqrt
{1+\rho^{2}}}{\tilde{\kappa}^{2}(1+\tilde{H}^{2})}\text{.} \label{25}%
\end{equation}
and $s_{t}$ is a natural representation of the tangent indicatrix of the curve
$\alpha$. The geodesic curvature of the principal image of the principal
normal indicatrix of $\alpha_{t}(s)$ is given by
\begin{equation}
\Gamma_{t}=\frac{-\tilde{\kappa}^{3}(1+\tilde{H}^{2})^{\frac{3}{2}}%
	(\rho-\tilde{H})^{2}[\tilde{\kappa}^{^{\prime\prime}}\tilde{\kappa}%
	(1+\tilde{H}^{2})-3\tilde{\kappa}^{^{\prime}2}(1+\rho\tilde{H})]}{\sqrt
	{1+\rho^{2}}(\tilde{\kappa}(1+\tilde{H}^{2})^{3}+\tilde{\kappa}^{^{\prime}%
		2}(\tilde{H}-\rho)^{2})^{\frac{3}{2}}}\frac{ds^{\ast}}{ds_{t}} \label{26}%
\end{equation}
where $\frac{ds^{\ast}}{ds_{t}}=\frac{\sqrt{1+\rho^{2}}}{\tilde{\kappa}%
	(\tilde{H}-\rho)}.$
\end{theorem}
After these computations, the following theorem can be written:

\begin{theorem}
	Suppose that $(\alpha,\tilde{\alpha})$ is a non-helical and non-planar
	Bertrand curve pair in three dimensional Lie group.Then the following
	properties hold:
	
	$i)$ If the curve $\alpha$\ is a slant helix, then its necessary and
	sufficient condition is $\alpha_{t}=$ $spherical$ $helix$.
	
	$ii)$ If the curve $\tilde{\alpha}$ is a slant helix, then its necessary and
	sufficient condition\ is\ $\alpha_{t}=$ $spherical$ $helix$.
\end{theorem}

Further the following theorem can be given:
\begin{theorem}
	Suppose that $(\alpha,\tilde{\alpha})$ is a non-helical and non-planar
	Bertrand curve pair in three dimensional Lie group. If the tangent indicatrix
	of $\alpha$\ is a spherical helix, then its necessary and sufficient condition
	is%
	\begin{equation}
	\tilde{\kappa}^{^{\prime\prime}}\tilde{\kappa}(1+\tilde{H}^{2})-3\tilde
	{\kappa}^{^{\prime}2}(1+\rho\tilde{H})=0\text{.} \label{27}%
	\end{equation}
	
\end{theorem}
\subsection{Principal normal indicatrix $\alpha_{n}$ $=N(s)$ of the	Bertrand curve in $G$}
\begin{definition}
	Suppose that the curve $\alpha$ is a regular curve in $G$. The curve
	$\alpha_{n}$ is called the principal normal indicatrix of the curve $\alpha$
	and $\alpha_{n}:I\subset S^{2}\subset g$ is\ defined by%
	\begin{equation}
	\alpha_{n}(s_{n})=N(s)\text{.} \label{28}%
	\end{equation}
	
\end{definition}

From here, we can write the principal normal indicatrix $\alpha_{n}$ $=N(s)$
as follows:
\[
\alpha_{n}=\epsilon\tilde{N}\text{.}%
\]

\begin{theorem}
	Suppose that the curve $\alpha$ is a regular curve in $G$ and the Frenet
	apparatus of the normal indicatrix $\alpha_{n}$ $=\epsilon\tilde{N}$ of the
	Bertrand curve are denoted by $\{T_{n},N_{n},B_{n},\kappa_{n},(\tau-\tau
	_{G})_{n}\}$. Then we have%
	\begin{align}
	T_{n}  &  =\frac{-\epsilon\{\tilde{T}-\tilde{H}\tilde{B}\}}{\sqrt{1+\tilde
			{H}^{2}}}\text{, }\label{29}\\
	N_{n}  &  =\frac{-\epsilon\{\tilde{H}\tilde{\kappa}^{^{\prime}}(\rho-\tilde
		{H})\tilde{T}-\tilde{\kappa}^{2}(1+\tilde{H}^{2})^{2}\tilde{N}-\tilde{\kappa
		}^{^{\prime}}(\tilde{H}-\rho)\tilde{B}\}}{\sqrt{\tilde{\kappa}^{^{\prime}%
			2}(\tilde{H}-\rho)^{2}+\tilde{\kappa}^{4}(1+\tilde{H}^{2})^{3}}\sqrt
	{1+\tilde{H}^{2}}}\label{30}\\
B_{n}  &  =\frac{\tilde{\kappa}^{2}\tilde{H}(1+\tilde{H}^{2})\tilde{T}%
	+\tilde{\kappa}^{^{\prime}}(\rho-\tilde{H})\tilde{N}+\tilde{\kappa}%
	^{2}(1+\tilde{H}^{2})\tilde{B}}{\sqrt{\tilde{\kappa}^{^{\prime}2}(\tilde
		{H}-\rho)^{2}+\tilde{\kappa}^{4}(1+\tilde{H}^{2})^{3}}}\text{\ } \label{31}%
\end{align}
and
\begin{align}
s_{n}  &  =%
{\displaystyle\int}
\frac{\tilde{\kappa}(\rho-\tilde{H})\sqrt{1+\tilde{H}^{2}}}{\tilde{H}%
	\sqrt{1+\rho^{2}}}ds\text{, }\kappa_{n}=\frac{\sqrt{\tilde{\kappa}^{^{\prime
			}2}(\tilde{H}-\rho)^{2}+\tilde{\kappa}^{4}(1+\tilde{H}^{2})^{3}}}%
{\tilde{\kappa}^{2}(1+\tilde{H}^{2})^{\frac{3}{2}}}\text{, }\label{32}\\
(\tau-\tau_{G})_{n}  &  =\frac{\epsilon(\tilde{H}-\rho)[(3\tilde{\kappa
	}^{^{\prime}2}-\tilde{\kappa}^{^{\prime\prime}}\tilde{\kappa})(1+\rho
	^{2})-3\tilde{\kappa}^{^{\prime}2}\rho(\tilde{H}-\rho)]}{\tilde{\kappa
	}^{^{\prime}2}(\tilde{H}-\rho)^{2}+\tilde{\kappa}^{4}(1+\tilde{H}^{2})^{3}%
}\text{.} \label{33}%
\end{align}
and $s_{n}$ is a natural representation of the principal normal indicatrix of
the curve $\alpha$.
\end{theorem}

Further the following theorem can be written:

\begin{theorem}
	Suppose that $(\alpha,\tilde{\alpha})$ is a non-helical and non-planar
	Bertrand curve pair in three dimensional Lie group. If the the principal
	normal indicatrix of $\alpha$\ is planar, then its necessary and sufficient
	condition is%
	\begin{equation}
	(3\tilde{\kappa}^{^{\prime}2}-\tilde{\kappa}^{^{\prime\prime}}\tilde{\kappa
	})(1+\rho^{2})-3\tilde{\kappa}^{^{\prime}2}\rho(\tilde{H}-\rho)=0\text{.}
	\label{34}%
	\end{equation}
\end{theorem}
\subsection{Binormal indicatrix $\alpha_{b}$ $=B(s)$ of the Bertrand curve in $G$}
\begin{definition}
	Suppose that the curve $\alpha$ is regular curve in $G$. The curve $\alpha
	_{b}$ is called the binormal indicatrix of the curve $\alpha$ and $\alpha
	_{b}:I\subset S^{2}\subset g$ is defined by%
	\begin{equation}
	\alpha_{b}(s_{b})=B(s)\text{.} \label{35}%
	\end{equation}
	
\end{definition}

Thus, the binormal indicatrix of the Bertrand curve in $G$ is given by%
\begin{equation}
\alpha_{b}=\frac{-\epsilon}{\sqrt{1+\rho^{2}}}\{\rho\tilde{T}+\tilde
{B}\}\text{.} \label{36}%
\end{equation}

\begin{theorem}
	Suppose that the curve $\alpha$ is a regular curve in $G$ and the Frenet
	apparatus of the binormal indicatrix $\alpha_{b}$ $=B$ of the Bertrand curve
	are denoted by $\{T_{b},N_{b},B_{b},\kappa_{b},(\tau-\tau_{G})_{b}\}$.Then we
	have%
	\begin{align}
	T_{b}  &  =\epsilon\tilde{N}\label{37}\\
	N_{b}  &  =\frac{-\epsilon}{\sqrt{1+\tilde{H}^{2}}}\{\tilde{T}-\tilde{H}%
	\tilde{B}\}\label{38}\\
	B_{b}  &  =\frac{1}{\sqrt{1+\tilde{H}^{2}}}\text{\ }\{\tilde{H}\tilde
	{T}+\tilde{B}\} \label{39}%
	\end{align}
	and%
	\begin{equation}
	s_{b}=-%
	{\displaystyle\int}
	\frac{\tilde{\kappa}(\rho-\tilde{H})^{2}}{\tilde{H}(1+\rho^{2})}ds\text{,
	}\kappa_{b}=\frac{\sqrt{(1+\tilde{H}^{2})(1+\rho^{2})}}{\rho-\tilde{H}}\text{,
}(\tau-\tau_{G})_{b}=\frac{\epsilon\tilde{\kappa}^{^{\prime}}\sqrt{1+\rho^{2}%
}}{\tilde{\kappa}^{^{2}}(1+\tilde{H}^{2})}\text{.} \label{40}%
\end{equation}
and $s_{b}$ is a natural representation of the binormal indicatrix of the
curve $\alpha$.
\end{theorem}

The geodesic curvature of the principal image of the principal normal
indicatrix of $\alpha_{b}(s)$ is given by%
\begin{equation}
\Gamma_{b}=\frac{-\tilde{\kappa}^{3}(1+\tilde{H}^{2})^{\frac{3}{2}}%
	(\rho-\tilde{H})^{2}[\tilde{\kappa}^{^{\prime\prime}}\tilde{\kappa}%
	(1+\tilde{H}^{2})-3\tilde{\kappa}^{^{\prime}2}(1+\rho\tilde{H})]}{\sqrt
	{1+\rho^{2}}(\tilde{\kappa}(1+\tilde{H}^{2})^{3}+\tilde{\kappa}^{^{\prime}%
		2}(\tilde{H}-\rho)^{2})^{\frac{3}{2}}}\frac{ds^{\ast}}{ds_{b}} \label{41}%
\end{equation}
where $\frac{ds^{\ast}}{ds_{b}}=\frac{\sqrt{1+\rho^{2}}}{\tilde{\kappa}%
	(\tilde{H}-\rho)}$.

From Theorem 3.11, we obtain the following corollaries:

\begin{corollary}
	Suppose that the curve $\alpha$ is \ a Bertrand curve in $G.$ Since
	$\Gamma_{t}=\Gamma_{b},$ the spherical images of the tangent and binormal
	indicatrices of $\alpha$ are the curves with same curvature and same torsion.
\end{corollary}

\begin{corollary}
	Suppose that $(\alpha,\tilde{\alpha})$ is a non-helical and non-planar
	Bertrand curve pair in $G.$ If the curve $\alpha$\ is a slant helix, then its
	necessary and sufficient condition is $\alpha_{b}=$ $spherical$ $helix$.
\end{corollary}

\begin{corollary}
	Suppose that $(\alpha,\tilde{\alpha})$ is a non-helical and non-planar
	Bertrand curve pair in $G.$ If the curve $\tilde{\alpha}$\ is a slant helix,
	then its necessary and sufficient condition is $\alpha_{b}=$ $spherical$
	$helix$.
\end{corollary}


\begin{center}

\end{center}

{\small

\noindent{\bf Ali \c{C}akmak}

\noindent Department of Mathematics,

\noindent Faculty of Arts and Sciences

\noindent Bitlis Eren University

\noindent Bitlis,Turkey

\noindent E-mail: acakmak@beu.edu.tr}\\

\end{document}